\documentclass[11pt, letterpaper]{amsart}   	
\usepackage{graphicx}				
\usepackage{amssymb}

\usepackage{latexsym,exscale,enumerate,amsfonts,amssymb,mathtools}
\usepackage{amsmath,amsthm,amsfonts,amssymb,amscd,stmaryrd,textcomp}
\usepackage[normalem]{ulem}
\usepackage{young}
\usepackage{thmtools,xcolor}
\usepackage{easybmat}
\usepackage{tikz-cd}

\definecolor{colormy}{rgb}{0.8,0.05,0.05}
\definecolor{mycolor}{rgb}{0.25,0.99,0.25}

\usepackage[all]{xy}
\SelectTips{cm}{}

\usepackage{tikz}
\usetikzlibrary{decorations.markings}
\usetikzlibrary{decorations.pathreplacing}
\usetikzlibrary{arrows,shapes,positioning}
\tikzstyle directed=[postaction={decorate,decoration={markings,
    mark=at position #1 with {\arrow{>}}}}]
\tikzstyle rdirected=[postaction={decorate,decoration={markings,
    mark=at position #1 with {\arrow{<}}}}]

\usepackage{hyperref}

\newcommand{\Hom}{\mathrm{Hom}}

\newcommand{\Ext}{\mathrm{Ext}}
\newcommand{\Tor}{\mathrm{Tor}}

\newcommand{\Char}{\mathrm{ch}}

\newcommand{\spa}{\mathrm{span }}

\newcommand{\Soc}{\mathrm{Soc}}
\newcommand{\Hd}{\mathrm{Hd}}

\def\Z{{\mathbb Z}}

\def\Q{{\mathbb Q}}

\def\P{{\mathbb P}}

\def\Ind{\mathrm{Ind}}

\def\cha{\mathrm{char}}

\theoremstyle{definition}
\newtheorem{thm}{Theorem}[section]
\newtheorem{cor}[thm]{Corollary}
\newtheorem{lem}[thm]{Lemma}
\newtheorem{prop}[thm]{Proposition}

\theoremstyle{definition}

\newtheorem{problemcounter}[thm]{Problem}

\numberwithin{equation}{section}

\declaretheorem[style=definition,name=Definition,qed=$\blacktriangle$,numberlike=thm]{defn}
\declaretheorem[style=definition,name=Remark,qed=$\blacktriangle$,numberlike=thm]{rem}

\title{Representation theory via cohomology of line bundles}

\author{Henning Haahr Andersen}
\address{Centre for Quantum Geometry (QM), Imada,
University of Southern Denmark, Odense,  Denmark}
\email{h.haahr.andersen@gmail.com}

\date{}							

\begin{document}

\maketitle

{\centering \it{Dedicated to the memory of Jim Humphreys}\par}

\begin{abstract}
Let $G$ be a reductive algebraic group over a field $k$ and let $B$ be a Borel subgroup in $G$. We demonstrate how a number of results on the cohomology of line bundles on the flag manifold $G/B$ have had interesting consequences in the representation theory for $G$. And vice versa.  Our focus is on the case where the characteristic of $k$ is positive. In this case both the vanishing behavior of the cohomology modules for a line bundle on $G/B$ and the $G$-structures of the non-zero cohomology modules are still very much open problems. We give an account of the developments over the years, trying to illustrate what is now known and what is still not known today.

\end{abstract}

\section{Introduction}
Let $G$ be a connected reductive algebraic group over an algebraically closed field $k$ and let $B$ be a Borel subgroup in $G$. Then any finite dimensional $B$-module $E$ induces a vector bundle (locally free sheaf) $\mathcal L(E)$ on the homogeneous space $G/B$ as follows: For any open subset $U \subset G/B$  the set of sections of  $\mathcal L(E)$ over $U$ is
\[ \Gamma(U, \mathcal L(E)) = \{\varphi:\pi^{-1}(U) \rightarrow E \mid \varphi (xb) = b^{-1} \varphi(x), x \in \pi^{-1}(U), b \in B \}
\]
Here $\pi$ denotes the canonical map $G \to G/B$, and the maps $\varphi$ considered are the regular maps from the open subvariety $\pi^{-1}(U)$ of the affine variety $G$ to the affine space $E \simeq k^n$.

In particular, the space of global sections $\Gamma (G/B,\mathcal L(E))$ has a natural $G$-action given by $g \varphi : x \mapsto \varphi (g^{-1} x), g,x \in G, \varphi \in \Gamma((G/B, \mathcal L(E))$. This  is also the $G$-module obtained by applying the induction functor $\Ind_B^G$ to $E$. More generally, we can for any $i \geq 0$ identify the sheaf cohomology module $H^i(G/B, \mathcal L(E))$ with the module obtained by applying the right derived functor $R^i\Ind_B^G$ to $E$, cf. \cite{RAG}, Proposition I.5.12.  We denote this module $H^i(E)$ for short. 

The aim of this paper is to give - with all deliberate hindsight - an account of the interplay between the representation theory for $G$ and the study of the cohomology modules $H^i(E)$. A major role in this study is played by the line bundles, i.e. the case where $E$ is a $1$-dimensional $B$-module. As will become clear there have been interactions in both directions : results in the representation theory have been obtained from investigating the cohomology of line bundles on $G/B$, and representation theoretic results have influenced the computations of cohomology modules. Key examples are Borel-Weil-Bott theory, the strong linkage principle,  Kempf's vanishing theorem, Frobenius splitting, Jantzen type sum formulas.

We begin our account by the Borel-Weil-Bott theory dating back to the early 1950's and follow up by giving some of the important developments - especially in modular representation theory - since then. Of course we can only present a tiny bit of the total work done over this 70-year period. Moreover, our choices are highly biased towards our own interests and our own contributions. Fortunately, there already exists several surveys and books in the area with broader perspectives.  In particular, the well known book by Jantzen, \cite{RAG}, gives an extensive account of many developments in the field (up until its publication in 2003). 

In our treatment here we have included some proofs and given references for the remaining ones. The proofs we give will often contain simplifications of the original proofs. In other cases we have added further details as we have found appropriate.  
The paper also contains a number of remarks, which we hope put the developments into perspective.  And we mention several problems, which are still open and waiting to be explored.

A famous underlying challenge, namely the problem of determining the irreducible characters for $G$ in characteristic $p>0$, has been a key motivating factor and driving force behind a lot of the work we discuss in this survey. However, we have not included anything on the marvelous recent breakthrough on this problem (and on the related a priori even harder looking problem of determining the characters of all indecomposable tilting modules for $G$). For this story we refer to \cite{AMRW} and \cite{RW}, cf. also the further references in these papers.

A look at the list of publications by J.E. Humphreys reveals that he - over the major part of the period we are treating -  had a keen interest in exploring the relations between representation theory and cohomology  of line bundles on $G/B$. His very last mathematical note, \cite{Hu18} concerned these topics. In fact, Jim's questions, results, suggestions, and conjectures over the years have strongly influenced a lot of the work done in this area. This certainly includes my own work, beginning when I was a graduate student  at MIT in 1975-77.  I'm very happy to dedicate this paper to his memory.

\section{The Borel-Weil and Borel-Weil-Bott theorems}

The original formulations of the two theorems discussed in this section originally concerned compact complex Lie groups. In our formulation we have stated the results for connected reductive algebraic groups $G$ (using the notation as in the introduction). As we shall point out more explicitly below, the theorems are strictly characteristic $0$ results. 

We shall need a little more notation. Choose a maximal torus $T$ contained in $B$ and let $X = X(T)$ denote the character group for $T$. If $\lambda \in X$ we extend it to $B$ by letting it be trivial on the unipotent radical $U \subset B$. Inside $X$ we have the root system $R$ for $(G,T)$ and we choose the the set of positive roots $R^+$ to be those of the Borel subgroup opposite to $B$. The set of simple roots in $R^+$ we denote by $S$, and the set of dominant weights $X^+$ is the set of $\lambda \in X$ for which $\langle \lambda, \alpha^\vee \rangle \geq 0$ for all $\alpha \in S$.
The Weyl group $N_G(T)/T$ for $G$ is denoted $W$. It acts naturally on $X$. In both this and the rest of this paper we will also use the ``dot''-action of $W$ on $X$ given by
$$ w \cdot \lambda = w(\lambda + \rho) - \rho, \lambda \in X, w \in W.$$
Here $\rho$ is half the sum of the positive roots. For convenience we assume $\rho \in X$.

If $M$ is a $T$-module and $\lambda \in X$ we define the $\lambda$-weight space in $M$ by
$$ M_\lambda = \{m \in M \mid t m = \lambda(t) m, \;t \in T\}.$$
We say that $\lambda$ is a weight of $M$ if $M_\lambda \neq 0$. The dimension of $M_\lambda$ is called the multiplicity of $\lambda$ as a weight of $M$.

Finally, the length function on $W$ corresponding to the set of simple roots in $R^+$ is denoted $\ell$.

\subsection{The Borel-Weil theorem}

\begin{thm} \label{B-W}(The Borel-Weil theorem, \cite{Serre}) Assume $k$ has characteristic $0$. Then for each $\lambda \in X^+$ the $G$-module $H^0(\lambda)$ is irreducible. Moreover, any finite dimensional $G$-module is isomorphic to $H^0(\lambda)$ for a unique $ \lambda \in X^+$.
\end{thm}

\begin{rem}
Suppose $\cha (k) = p >0$. Then the analogue of Theorem \ref{B-W} is false already for $G= SL_2$. In fact, in that case $X^+ = \Z_{\geq 0}$ and the module $H^0(\lambda)$ is only irreducible for very special values of  $\lambda$, namely $\lambda  = ap^n -1$ for some $0 < a < p$ and $n \geq 0$. 
\end{rem}

However, in all characteristics we have the following classification of the finite dimensional irreducible $G$-modules. We use the notation $w_0$ for the longest element in $W$.
\begin{thm} \label{Chev} (Chevalley, see \cite{Chev})
For each $ \lambda \in X^+$ the $G$-module $H^0(\lambda)$ has a unique irreducible submodule $L(\lambda)$, and if $L $ is an arbitrary finite dimensional irreducible $G$-module then $L$ is isomorphic to $L(\lambda)$ for a unique $\lambda \in X^+$. Moreover, all weights $\mu$ of $H^0(\lambda)$, respectively $L(\lambda)$, satisfy $w_0\lambda \leq \mu \leq \lambda$, and the weight $\lambda$ occurs with multiplicity $1$. 

\end{thm}

\subsection{The Borel-Weil-Bott theorem}
We shall need the fact, that for any $\lambda \in X$ the intersection $(W \cdot \lambda) \cap X^+$ is either empty or equal to $\lambda^+$ for some unique $\lambda^+ \in X^+$. In the first case we say that $\lambda$ is singular, and in the second that $\lambda$ is regular. 

We can now formulate the Borel-Weil-Bott theorem (or sometimes just Bott's theorem) as follows

\begin{thm} \label{Bott} (Bott 1957, \cite{Bott}). Assume that $k$ has characteristic $0$. For $\lambda \in X$ we have
$$ H^i(\lambda) = \begin{cases} {H^0(\lambda^+) \text{ when $\lambda = w\cdot \lambda^+$ for some $w \in W$ and $i = \ell(w),$}} \\ {0   \text { otherwise. }} \end{cases}$$
\end{thm}

As indicated this result was obtained by R. Bott in 1957. Later M. Demazure gave first a simple proof \cite{D68} and then a very simple proof  \cite{D76}  of Bott's theorem. An alternative very simple proof can be found in \cite{An79b}, Remark 3.3.

We like to describe Bott's theorem as containing two parts. Firstly, it describes completely the vanishing behavior of all homogeneous line bundles on $G/B$: those corresponding to regular characters
 have exactly one non-vanishing cohomology group, and all cohomology vanish for the line bundles associated to singular characters. Secondly, it proves that each non-vanishing cohomology module is irreducible, and it singles out its highest weight. In characteristic $p>0$ the first part holds for $SL_2$ but for no simple group of higher rank. D. Mumford was the first to find an example of a line bundle with more than one non-vanishing cohomology module. His example was a line bundle on the $3$-dimensional flag variety $SL_3/B$. For this flag variety his student L. Griffith gave later a complete description of the vanishing behavoir of line bundles in his thesis, cf. \cite{Grif}. 
 
 As observed above the second part of Bott's theorem fails already for $SL_2$ (and hence for all other groups as well). This presents us with two problems:

\begin{problemcounter} Describe the vanishing behavior for the cohomology of all line bundles on $G/B$.
\end{problemcounter}

\begin{problemcounter}
Describe the $G$-module structures of the cohomology modules of all line bundles on $G/B$.
\end{problemcounter}

In the following sections we shall give some partial answers to these two problems. As will become clear complete answers still seem very much out of reach. However, the most important subproblem, namely the problem of determining the composition factors of $H^0(\lambda)$ for all $\lambda \in X^+$ has now been settled (in the sense that the composition factor multiplicities of $H^0(\lambda)$ is expressed in terms of so called $p$-Kazhdan-Lusztig polynomials) by the recent breakthrough by Riche and Williamson \cite{AMRW}, \cite{RW}  (for primes less than  $2h-2$ see also \cite{So}). 

Even though we have separated the problem of determining the cohomology modules $H^i(\lambda)$, $i \in \Z_\geq 0$ and $\lambda \in X$ into the two individual problems, Problems 2.5 and Problem 2.6, it will hopefully be clear from the following sections that they are very much interrelated and we advocate very strongly to explore them together. For instance, my proof of the strong linkage principle (see  Section 5.1 below) only succeeded when I decided to explore the structure of the cohomology modules without knowing their vanishing behavior.

\section{Kempf's vanishing theorem}

A specially important part of Problem 2.5 is to describe the cohomology of dominant line bundles, i.e.  line bundles induced by dominant characters. Note that in characteristic zero Bott's theorem gives that all higher cohomology of such line bundles vanishes. Moreover, it is easy to see that if $\lambda$ is strictly dominant, i.e. $\lambda - \rho \in X^+$, then $\mathcal L(\lambda)$ is an ample line bundle on $G/B$. Hence for such a character we have in all characteristics that $H^i(n \lambda) = 0$ for $i >0$ whenever $n \gg 0$. This led to the expectation that the vanishing of the higher cohomology of all line bundles with dominant weights would hold in all characteristics. This turned out to be true:

\begin{thm} \label{Kempf} (Kempf's vanishing theorem, \cite{Ke76b}) $\text {If } \lambda \in X^+ \text { then for all } p \text { we have } H^i(\lambda) = 0 \text { for all }  i > 0.$
\end{thm}

\begin{rem}  \label{rem on Kempf}
\begin{enumerate}
\item The importance of this theorem in modular representation theory is tremendous.  First of all, it immediately gives that for $\lambda \in X^+$ the character of $H^0(\lambda)$ is independent of the characteristic and hence given by the Weyl formula. Using Serre duality this implies that $H^N(w_0\cdot \lambda)$ is the Weyl module with highest weight $\lambda$. It also reveals that we have $\Z$-lattices, i.e. there is a module $H^0_\Z(\lambda)$, respectively $H^N_\Z(w_0\cdot \lambda)$  for the Chevalley group over $\Z$ corresponding to $G$, which is free over $\Z$ and has $H^0_Z(\lambda) \otimes_\Z k \simeq H^0(\lambda)$, respectively  $H^M_Z(w_0 \cdot \lambda) \otimes_\Z k \simeq H^N(w_0 \cdot \lambda)$. In particular, we see that to determine the irreducible characters, $\Char L(\lambda), \lambda \in X^+$ is equivalent to finding the composition factor multiplicities of all Weyl modules. All approaches to finding the irreducible characters in characteristic $p$ rely on this fact.

\item (on the proof of Kempf's theorem)
G. Kempf first proved this result for $G = SL_n$ for all $n$, \cite{Ke76a}. Then Lakshmi Bai, C. Musili and C.S. Seshadri extended the result to other classical groups, \cite{LMS}, before Kempf came up with his general proof. The method of proof in these papers are all algebraic geometric involving a close analysis of the restrictions of $\mathcal L(\lambda)$ to certain Schubert varieties in $G/B$ and $G/P$, $P$ a parabolic subgroup containing $B$.  As a side benefit this gave also important vanishing theorems for the higher cohomology of the restrictions of dominant line bundles to various Schubert varieties. 

\item (simple proofs and further work)
A few years after Kempf's general proof the author, see  \cite{An80b} (or Theorem \ref{F-S} below), and independently W. Haboush \cite{Ha}, came up with a much shorter proof of Theorem \ref{Kempf}. These proofs combine (well known) algebraic geometric facts, in particular the vanishing of higher cohomology of sufficiently high powers of ample line bundles, and the properties of the Frobenius morphism in prime  characteristics  with (equally well known) representation theoretic facts, especially the existence and properties of the Steinberg modules (see below). The author realized a few years \cite{An85} later that a slight twist of the arguments used in \cite{An80b} could be used to extend the vanishing result to all Schubert varieties in $G/B$. As corollaries this proved on the algebraic geometric side the normality of Schubert varieties and on the representation theoretic side the Demazure character formula. Results in the same directions had earlier been explored by M. Demazure, \cite {D74} in characteristic $0 $, and by V. Lakshmibai, C. Musili  and C.S. Seshadri, \cite{LMS} in general. In particular, it should be mentioned that C.S. Seshadri was the first to prove the normality of all Schubert varieties in all characteristics, see \cite{CSS}. 
\item Later, the Frobenius splitting method - invented and developed by Mehta and Ramanathan, \cite{MeRa} -  became a ``big industry'' with a large number of applications in both algebraic geometry and in representation theory. We do not report further on this line of research here but refer instead the reader to the book \cite{BK85}.
\end{enumerate}
\end{rem}

\section{On the vanishing behavior, Problem 2.5}
In this section we try to describe what is known and what is still not known about the vanishing of the cohomology modules $H^i(\lambda)$, $i \in \Z_{\geq 0}, \lambda \in X$. Throughout we assume (unless otherwise said) that the characteristic of $k$ is $p >0$.

\subsection{Preliminaries}
In the following $N$ will denote the set of positive roots. Alternatively, $N$ is the dimension of the flag manifold $G/B$.

By Grothendieck vanishing \cite{Gr}, Theorem 3.6.5 we have
\begin{equation} \label{Grot}  H^i(\lambda) = 0 \text { for all } \lambda \in X \text { when } i > N. \end{equation}

This theorem implies that for each $\lambda \in X$ we have only finitely many cohomology modules to worry about. 

We can further ``cut the vanishing problem in half'' by using Serre duality:
\begin{equation} \label{Serre} H^i(\lambda)^* \simeq H^{N-i}(-\lambda  - 2 \rho) \text { for all } 0 \leq i \leq N, \lambda \in X.
\end{equation}
Here, if $M$ is a $G$-module, we use the notation $M^*$  for the contragredient dual module. The isomorphism in (\ref{Serre}) is an isomorphism of $G$-modules. For later use we record also the following consequence
\begin{equation} \label{hom} \Hom_G(H^N(w_0 \cdot \lambda), H^0(\lambda)) \simeq k \text { for all } \lambda \in X^+. \end{equation}
In fact, by (\ref{Serre}) we have that $H^N(w_0 \cdot \lambda) \simeq H^0(-w_0 \lambda)^*$. Hence by Theorem \ref{Chev}  $\lambda$ is the highest weight of $H^N(w_0 \cdot \lambda)$, and it occurs with multiplicity $1$.  

\begin{rem} While we know of no representation theoretic proof of the Grothendieck vanishing theorem (\ref{Grot}), the Serre duality can be obtained ``purely representation-theoretically'' (and this proof generalizes to the 
quantum case), see \cite{AW}, Section 3.2.
\end{rem}

Let $n \geq 0$. The $n$'th Steinberg module is $St_n$. This  is the irreducible $G$-module with highest weight $(p^n-1)\rho$, i.e. 
$$ St_n = L((p^n-1)\rho). $$

We denote by $F: G \rightarrow G$ the Frobenius homomorphism on the group scheme $G$.  Its kernel is an infinitesimal group scheme denoted $G_1$. We define more generally  $G_n$  to be the kernel of $F^n$. If $V$ is an arbitrary $G$-module then $V^{(n)}$ denotes the $n$'th Frobenius twist of $V$. This means that $V^{(n)}$ as a vector space is identical to $V$ but its group action is the composite $ G \xrightarrow {F^{(n)}} G \rightarrow GL(V)$.  
We also define ``untwist'' of twisted modules: If $M = V^{(n)}$, i.e. if the restriction to $G_n$ of the $G$-action on $M$  is trivial, then we shall write $V = M^{(-n)}$.

A special case of Steinberg's tensor product theorem, \cite{St}, says
\begin{equation} \label{St-tensor} St_n = St_1 \otimes St_1^{(1)} \otimes \cdots \otimes St_1^{(n-1)}. 
\end{equation}

This implies that the Steinberg module is a (dual) Weyl module:
\[ St_n = H^0 ((p^n-1)\rho) \text { for all } n.\]

Note that by (\ref{St-tensor}) this statement reduces to the case $n=1$. In that case it follows from the strong linkage principle, \cite{An80a} (or see Theorem \ref{SLP} below).

Following \cite{An07} we shall set 
$$ D_p(i) = \{\lambda \in X \mid H^i(\lambda) \neq 0\}, \, i \geq 0, \, \lambda \in X.$$
(here we allow $p =0$).

Finally, we shall find it convenient to use the following notation 
$$ p^n\cdot \lambda = p^n (\lambda+\rho)-\rho, \;  n \geq 0, \, \lambda \in X,$$ 
and 
$$ X_n = \{\lambda \in X^+ \mid \langle \lambda, \alpha^\vee\rangle < p^n \text { for all simple roots } \alpha \}.$$
The elements in $X_n$ are called the $p^n$-restricted weights.

\subsection{The Frobenius-Steinberg theorem}

Using the above notation we have
\begin{thm} \label{F-S}(The Frobenius-Steinberg theorem, \cite{An80b}, Theorem 2.5).
Let $\lambda \in X$. Then we have $G$-module isomorphisms
$$H^i(p^n \cdot \lambda) \simeq St_n \otimes H^i(\lambda)^{(n)}$$ for all $i, n \geq 0$.
\end{thm}

Note that here we have stated this theorem only for line bundles. It actually holds for all vector bundles on $G/B$ (and is stated and proved in this generality in \cite{An80b}). Let us point out that the proof follows easily from the following standard facts (about induction to and from the group scheme $G_nB$, respectively about the Steinberg modules).

\begin{proof}
Let $E$ be a $B$-module. Then clearly $E^{(n)}$ is a $G_nB$-module and we have
\begin{equation} \label{higher vanishing} H^i(G_nB/B, E) = 0 \text { for all } i > 0,
\end{equation} 
and 
\begin{equation} \label{cohomology of twist}  H^i(E)^{(n)} \simeq H^i(G/G_nB, E^{(n)})  \text {  for all } i\geq 0.
\end{equation}
Let $ \lambda \in X$ and write as before $\lambda = \lambda^0 + p^n\lambda^1$ with $\lambda^0 \in X_n$ and $\lambda^1 \in X$. Then we have an isomorphism of $G_nB$-modules 
\begin{equation} \label{ind to G_nB} H^0(G_nB/B, \lambda) \simeq H^0(G_nB/B, \lambda^0) \otimes p^n\lambda^1.
\end{equation}
Moreover, in the special case $\lambda = (p^n -1) \rho$ we have a $G_nB$-isomorphism
\begin{equation} \label{ind and Steinberg}  H^0(G_nB/B, (p^n-1)\rho) \simeq {St_n}_{|_{G_nB}}.
\end{equation}

\vskip .2 cm

To obtain the Frobenius-Steinberg theorem (in the form stated above) from these facts we first use (\ref{higher vanishing}) to see that 
$$ H^i(p^n \cdot \lambda) = H^i( (p^n-1)\rho + p^n\lambda) \simeq H^i(G/G_nB, H^0(G_nB/B, (p^n-1)\rho + p^n\lambda)).$$
Here according to (\ref{ind to G_nB}) and (\ref{ind and Steinberg}) the term $H^0(G_nB/B, (p^n-1)\rho + p^n\lambda)$ is isomorphic as $G_nB$-module to $ St_n \otimes p^n\lambda$. So using this and the (generalized) tensor identity  \cite{RAG}, Proposition I.3.6 we get
$$  H^i(G/G_nB, H^0(G_nB/B, (p^n-1)\rho + p^n\lambda)) \simeq St_n \otimes H^i(G/G_nB, p^n\lambda).$$
We conclude by applying (\ref{cohomology of twist}).
\end{proof}

The first and foremost consequence of Theorem \ref{F-S} is the Kempf vanishing theorem for the higher cohomology of dominant line bundles, see Remark \ref{rem on Kempf} (2). However, the result certainly has important applications also for non-dominant line bundles, including the vanishing behavior of their cohomology. 
 The most evident application is the following

\begin{cor} \label{Dot}  Let $i, n \geq 0$ and $\lambda \in X$. Then $\lambda \in X_p(i)$ if and only if $p^n \cdot \lambda \in D_p(i)$. 
\end{cor}

With a little more effort we can prove, see \cite{An80b}, Proposition 3.3 (using the same notation as in Corollary \ref{Dot})
\begin{equation}  \label{stable-} \lambda \in D_p(i) \text{  if and only if  } p^n \cdot \lambda - X_n \subset D_p(i). \end{equation}

The corollary says that $D_p(i)$ is stable under dot multiplication by $p^n$.  As $(p^n -1)\rho \in X_n$ we see from (\ref{stable-}) that $D_p(i)$ is also stable under usual multiplication by $p^n$. This last result is also a consequence of the following result.

\begin{cor} \label{Frob-inj} (\cite{An80b}, Corollary 2.7)
Let $\lambda \in X$. Then the Frobenius homomorphism on $G$ induces injective homomorphisms $H^i(\lambda)^{(n)} \rightarrow H^i(p^n \lambda)$ for all $i, n \geq 0$.
\end{cor}

\begin{rem} This corollary holds more generally for all vector bundles on $G/B$. It was conjectured by Cline, Parshall and Scott. I must admit that I didn't believe their conjecture when CPS mentioned it to me at an Oberwolfach meeting in April 1979, but my efforts of finding a counterexample resulted in the discovery of the Frobenius-Steinberg theorem above. Thus
 I ended up proving the conjecture instead!
\end{rem}

Combining Corollary \ref{Frob-inj} with Serre duality gives the analogous statement:
\begin{equation} \label{stable+} \lambda \in D_p(i) \text { if and only if } p^n \cdot \lambda + X_n \subset D_p(i). \end{equation}

Now we note that by semi-continuity (alternatively use the universal coefficient theorem, Theorem \ref{UC}) we have inclusions
$$ D_0(i) \subset D_p(i) \text { for all } i \geq 0.$$
By  Bott's theorem we know
$$D_0(i) = \bigcup_{w\in W,  \ell(w) = i} w \cdot X^+.$$

Combining these two facts with (\ref{stable-}) and (\ref{stable+}) we obtain the following result.

\begin{prop} \label{vanishing behavior} For all $i \geq 0$ we have 
$$ \bigcup_{n\geq 0,\; \ell(w) =i} (p^n\cdot w \cdot X^+ \pm X_n) \subset D_p(i).$$
\end{prop}

\begin{rem} 
\begin{enumerate}
\item It is well known that $D_p(0) = X^+$ for all $p \geq 0$. By Serre duality we then also have $D_p(N) = X^-$ where the set of antidominant weights $X^-$ is defined by $X^- = \{\lambda \in X \mid -\lambda -2\rho \in X^+\} = w_0 \cdot X^+$. It follows that the inequality in Proposition \ref{vanishing behavior} is an equality when $i = 0$ and $i = N$.
\item It will follow from the results in the next section that we have equality in  Proposition \ref{vanishing behavior} also for $i = 1$ and $i=N-1$. However, the results in \cite{An81} and \cite{AKa} show that equality does not hold in general. In fact, we expect it to fail for all $1 < i <N-1$ (by the computations in \cite{An81}, Section 5, and \cite{AKa}, it does so for type $B_2$ and $G_2$). Nevertheless,
 Proposition \ref{vanishing behavior} contains the best approximation to the (non-)vanishing behavior of the cohomology  of line bundles known up to this date.
\end{enumerate}
\end{rem}

\subsection{Non-vanishing of the first cohomology module}
We shall now describe those $\lambda \in X$ for which $H^1(\lambda) \neq 0$. In other words we determine the set $D_p(1)$ from Section 4.2. This result was obtained in \cite{An79b}, Theorem 3.6.a . It thus predates the Frobenius-Steinberg theorem above. Its proof is based on a detailed examination of the $B$-module structure of $H^1(P_\alpha/B, \mathcal L (\lambda))$, where $\alpha$ is a simple root and $P_\alpha$ is the minimal parabolic subgroup containing $B$ corresponding to $\alpha$. The proof also gives information about the $G$-module structures of the non-zero $H^1(\lambda)$, see Section 5.3.

 \begin{thm}\label{H1} Let $\lambda \in X$. Then $H^1(\lambda) \neq 0$ if and only if $\lambda \in \bigcup_{\alpha \in S, \; n\geq 0} (p^n \cdot s_\alpha \cdot X^+ - X_n) $.
 \end{thm}

\begin{rem} 
\begin{enumerate}
\item
The formulation of the non-vanishing criteria above is somewhat different from the way it is stated in \cite{An79b}. However, it is easy to check that it is equivalent. We have chosen the alternative formulation in order to stay close to the formulation in Section 4.2. The theorem says that for $i=1$ we have equality in Proposition  \ref{vanishing behavior} (note that in the $i = 1$ case we can replace the $\pm$ sign in this proposition with the minus sign only because $p^n \cdot s_\alpha \cdot \lambda + X_n \subset s_\alpha \cdot X^+$ for all $\lambda \in X^+$).
\item This theorem represents the first progress on describing the vanishing behavior of cohomology of non-dominant line bundles on $G/B$ for a general reductive algebraic group. The only known previous result on this question was Griffith's treatment of  the case $G = SL_3$,  \cite{Grif}. Except for Theorem \ref{H1}, the consequences of the Frobenius-Steinberg theorem presented in Section 4.2, and the computations in the rank $2$ cases, \cite{An81} and \cite{AKa}, the Problem 2.5 is still wide open for the cohomology modules $H^i(\lambda)$ with $1 < i < N-1$ and $\lambda \notin X^+ \cup X^-$.
\end{enumerate}
\end{rem}

The first ingredient in the proof of Theorem \ref{H1} is the Kempf vanishing theorem. It implies that if $\lambda \in D_p(1)$ then there exists a simple root $\alpha$ with $\langle\lambda, \alpha^\vee \rangle < 0$. This implies that the restriction of $\mathcal L(\lambda)$ to $P_\alpha/B \simeq \P^1$
has no cohomology in degrees different from $1$. Therefore, we get $H^1(\lambda) \simeq H^0(H^1(P_\alpha /B, \mathcal L(\lambda)))$. The proof now comes down to a close inspection of the $B$-module structure of $H^1(P_\alpha/B, \mathcal L(\lambda))$. For details we refer to the proof in \cite{An79b}.

Using Serre duality Theorem \ref{H1} is equivalent to
\begin{cor} 
Let $\lambda \in X$. Then $H^{N-1}(\lambda) \neq 0$ if and only if $\lambda \in \bigcup_{\alpha \in S, \; n\geq 0} (p^n \cdot s_\alpha \cdot X^- + X_n) $.
\end{cor}

Again this is equivalent to saying that for $i= N-1$ we have equality in Proposition \ref{vanishing behavior}.

\section{On the $G$-module structure. Problem 2.6}

Bott's theorem is a beautiful and complete description of all cohomology modules of line bundles on $G/B$. It gives a full account of both their vanishing behavior and of their G-structure. In a sense it is at the same time a disappointment: the higher cohomology modules for non-dominant line bundles present nothing new. In fact,  in characteristic zero any non-vanishing higher cohomology module is isomorphic to the (simple) $0$'th cohomology module of the corresponding $W$-dot-conjugated dominant line bundle.

In characteristic $p$ there is much more excitement. As we shall demonstrate in this section, it is for any $p > 0$  and $\mu \in X$ non-dominant a rare exception for $H^i(\mu)$ to be isomorphic to some $H^0(\lambda)$ with $\lambda$ dominant
. Moreover, when $1 < i < N-1$ it is not always true that $H^i(\mu)$ has simple socle or head. In fact, there are examples where it turns out that $H^i(\mu)$ is decomposable. In other words, just as Sections 4.2-3 revealed that the vanishing behavior is much more intricate in positive characteristics than in characteristic zero, so we shall see in this section is the situation about the $G$-module structures.

Despite this somewhat mysterious and to a large extent still unknown behavior of the $G$-modules $H^i(\mu)$, we shall start out by proving, that we can take good advantage of these higher cohomology modules to obtain not only some definite results about them,  but also in fact some general results in representation theory.

\subsection{The strong linkage principle}

Let $M$ be a finite dimensional $G$-module. For $\mu \in X^+$ we denote by $[M:L(\mu)]$ the composition factor multiplicity of $L(\mu)$ in $M$. Then we can state the strong linkage principle as follows.
\begin{thm} (The strong linkage principle, \cite{An80a}) \label{SLP}
Let $\lambda \in X$ and $\mu \in X^+$. If $[H^i(\lambda): L(\mu)] \neq 0$ for some $i \geq 0$, then $\mu \uparrow \lambda^+$.
\end{thm}
Here $\lambda^+$ is the unique element in $W \cdot \lambda \cap (X^+-\rho)$. Moreover, we have used the notation $\mu \uparrow \lambda^+$ from Jantzen's book, \cite{RAG}, II.6 to mean that $\mu$ is strongly linked to $\lambda^+$, i.e. there exists a sequence $\mu = \mu_0 \leq \mu_1 \leq \cdots \leq \mu_r = \lambda^+$ in $X$ with $\mu_{j-1} = s_{\beta_j} \cdot \mu_j + n_j p \beta_j$ for some $\beta_j \in R^+$ and $n_j \in \Z, j=1, 2, \cdots r$.

\vskip .5 cm
Before we give the proof of this theorem  we shall say a little about its background. A linkage principle for modular representations of reductive algebraic groups was formulated by D.-n. Verma in \cite{Ve}. It says that two composition factors $L(\lambda)$ and $L(\mu)$ of an indecomposable module for $G$ always satisfy $\mu \in W_p \cdot \lambda$. Here $W_p$ denotes the affine Weyl group for $G$, i.e. the group generated by the affine reflections $s_{\beta, r}$ defined by
$$ s_{\beta,r} \cdot \nu = s_\beta \cdot \nu + rp \beta, \; \nu \in X,$$
where $\beta$ runs through all positive roots, and $r$ through all integers.

This linkage principle was proved by J.E. Humphreys for $p > h$, see \cite{Hu71}. V. Kac and B. Weisfeiler improved this result by showing that the linkage principle holds for all  $p$ not dividing the index of connection for $R$ , see  \cite{KW}. Also the work of R. Carter and G. Lusztig, \cite{CL}, implies the principle for type $A_n$ for all $p$. We refer to \cite{Hu76}, §3 for more details on the background of the linkage principle.

Actually, D.-n. Verma suggested in \cite{Ve}, Conjecture II p. 689, that a stronger principle should hold for Weyl modules. This was proved to be the case by Jantzen for $p \geq h$, first for type $A_n$ in \cite{Ja74}  and then for arbitrary types in \cite{Ja77}.  I named this principle (for all cohomology modules of line bundles on $G/B$) the strong linkage principle and proved it for all $p$ in \cite{An80a}. 

The linkage principle is a consequence of the strong linkage principle for Weyl modules, see Corollary \ref{LP} below.
\vskip .5 cm
J.E. Humphreys asked me some time in the late $1970$'s whether the strong linkage principle could be extended to the higher cohomology groups on line bundles. In  \cite{An79a} I had proved this for the cohomology modules of certain very special line bundles, namely those belonging to $p$-alcoves containing $-\rho$ in their closures (for such line bundles I proved in fact that Bott's theorem holds). In \cite{An79b} we proved, that whenever $H^1(\lambda) \neq 0$ it has a simple socle. This implies that at least  the (weak) linkage principle holds for the first cohomology modules.

In the Fall 1979 I tried then to extend the linkage principles (weak or strong)  to all the higher cohomology modules. The idea was to use various long exact sequences of such higher cohomology modules to deduce the principle from its validity for the $0$'th cohomology modules, i.e. the dual Weyl modules. Suddenly, I realized that not only would these sequences indeed give this extension, but by including the higher cohomology modules in the statement of the strong linkage principle, one could in fact give a short proof - valid for all $p$ and all cohomology - of the principle itself.

 We begin with an easy lemma. If $\alpha \in S $ we denote by  $P_\alpha$ the minimal parabolic subgroup in $G$ containing $B$ and having $\alpha$ as its unique positive root. If $\lambda \in X$ we write $H^i_\alpha (\lambda) = H^i(P_\alpha/B, \mathcal L(\lambda))$. 
\begin{lem} \label{homos} Let $ \lambda  \in X$,  $\alpha \in S$, and assume $\langle \lambda + \rho, \alpha^\vee \rangle \geq 0$.
 Then there is for each $i \geq 0$ a natural $G$-homomorphism
$$ c_\alpha^i (\lambda): H^{i+1}(s_\alpha \cdot \lambda) \rightarrow  H^i(\lambda).$$

\end{lem}

\begin{proof}
As $P_\alpha/B \simeq \P^1$ and $ \langle \lambda, \alpha^\vee \rangle \geq 0$ we have  $H^j_\alpha(\lambda) = 0$ for all $j > 0$.
 By Serre duality we have moreover $H^1_\alpha( s_\alpha \cdot \lambda) \simeq H^0_\alpha (- s_\alpha (\lambda))^*$, and $H^j(s_\alpha \cdot \lambda) = 0$ for all $j \neq 1$.  It is then a matter of easy $SL_2$-computations to see that we have (up to scalars) a unique  $P_\alpha$-homomorphism $H^1_\alpha( s_\alpha \cdot \lambda) \rightarrow H^0_\alpha (\lambda)$. Via the Leray spectral sequence coming from the $P_\alpha/B \simeq \P^1$-fibration $G/B \rightarrow G/P_\alpha$ we get that this homomorphism induces  corresponding  natural homomorphisms $H^{i+1} ( s_\alpha \cdot \lambda) \rightarrow H^i(\lambda)$ for all $i \geq 0$. 
\end{proof}
\begin{rem} \label{wall-vanishing} Let $ \langle \lambda, \alpha^\vee \rangle = -1$. The above proof shows in particular that  for such $\lambda$ we have 
 $ H^j (\lambda) = 0$  for all $ j$. In fact, we get more generally for any $P_\alpha$-module $V$ 
  that  $H^j(V \otimes \lambda) = 0$ for all $j$ (as the vector bundle on $P_\alpha/B$ induced by $V$ is trivial).
\end{rem}

We shall now need $2$ long exact sequences involving our cohomology modules. They arise from $4$ short exact sequences of $B$-modules. Assume $\lambda \in X$ and $\alpha \in S $ satisfy $\langle \lambda, \alpha^\vee \rangle \geq 0$. 
Then the first two $B$-sequences are (recall that if $\mu \in X$ we also denote by $\mu$ the $1$-dimensional $B$-module defined by this character)
\begin{equation} \label{seq1} 0 \rightarrow K_\alpha^\lambda \rightarrow H^0_\alpha (\lambda + \rho) \otimes (-\rho) \rightarrow \lambda \rightarrow 0, \end{equation} 
and
\begin{equation} \label{seq2} 0 \rightarrow s_\alpha \cdot \lambda \rightarrow K_\alpha^\lambda \rightarrow V_\alpha^\lambda \rightarrow 0. \end{equation}
These sequences are obtained by noticing that $\lambda$, respectively $s_\alpha \cdot \lambda$, is the highest, respectively lowest, weight of 
 $H^0_\alpha (\lambda + \rho) \otimes (-\rho)$. This gives the surjection onto $\lambda$ in (\ref{seq1}), and we then define $K_\alpha^\lambda$ to be the kernel of this surjection. Likewise we get the injection in (\ref{seq2}) and define $V_\alpha^\lambda$ to be the corresponding cokernel.

We now note that by Remark \ref{wall-vanishing} the middle term in (\ref{seq1}) has vanishing cohomology in all degrees. Hence we get $H^j(\lambda) \simeq H^{j+1}(K_\alpha^\lambda)$ for all $j$. Inserting this in the long exact sequence coming from (\ref{seq2}) we get the long exact $G$-sequence
\begin{equation} \label{lseq1} \cdots \rightarrow H^{j+1}(s_\alpha \cdot \lambda) \rightarrow H^j(\lambda) \rightarrow H^{j+1}(V_\alpha^\lambda) \rightarrow \cdots \end{equation}
In this sequence the maps $H^{j+1}(s_\alpha \cdot \lambda) \rightarrow H^j(\lambda)$ are up to non-zero scalars equal to the homomorphisms $c_\alpha^j(\lambda)$ from Lemma \ref{homos}. 

To further explore the terms $ H^{j+1}(V_\alpha^\lambda) $ in the above sequence we shall need the following $2$ short exact $B$-sequences
\begin{equation} \label{seq3} 0 \rightarrow C_\alpha^\lambda \rightarrow V_\alpha^\lambda \rightarrow I_\alpha^\lambda \rightarrow 0, \end{equation} 
and
\begin{equation} \label{seq4} 0 \rightarrow I_\alpha^\lambda \rightarrow H^0_\alpha (\lambda+\rho - \alpha) \otimes (-\rho) \rightarrow Q_\alpha^\lambda \rightarrow 0. \end{equation}
These $2$ sequences come from the existence of a $B$-homomorphism (an easy $SL_2$-computation) $V_\alpha^\lambda \rightarrow  H^0_\alpha (\lambda+\rho - \alpha) \otimes (-\rho)$ by setting $C_\alpha^\lambda$, respectively $I_\alpha^\lambda$, respectively $Q_\alpha^\lambda$ equal to the kernel, respectively image, respectively cokernel, of this map.

\begin{rem} In characteristic $0$ we have  $V_\alpha^\lambda \simeq  H^0_\alpha (\lambda+\rho - \alpha) \otimes (-\rho)$. This is a key observation in \cite{D76} leading Demazure to his very easy proof of Bott's theorem. Its characteristic $p$ analogue plays likewise an important rôle in our proof here.
\end{rem}
The middle term in the sequence (\ref{seq4}) has vanishing cohomology in all degrees (again by Remark \ref{wall-vanishing}). Arguing as above we then obtain the long exact $G$-sequence
\begin{equation} \label{lseq2}
\cdots \rightarrow H^{j+1}(C_\alpha^\lambda) \rightarrow  H^{j+1}(V_\alpha^\lambda) \rightarrow H^{j}(Q_\alpha^\lambda) \rightarrow \cdots.
\end{equation}

Inspecting the homomorphism $V_\alpha^\lambda \rightarrow  H^0_\alpha (\lambda+\rho - \alpha) \otimes (-\rho)$ a bit closer we see that the set of weights of both $C_\alpha^\lambda$ and $
Q_\alpha^\lambda$ is 
$$X_\alpha^\lambda = \{s_\alpha \cdot \lambda + rp\alpha \mid 0 < rp <  \langle \lambda + \rho, \alpha^\vee \rangle \}.$$ This implies 
\begin{prop} \label{ker/cok} Let $\lambda$ and $\alpha$ be as above and pick $\mu \in X^+$. 
\begin{enumerate} \item If $L(\mu)$ is a composition factor of either the kernel or the cokernel of $c_\alpha^i(\lambda)$  for some $i$, then $L(\mu)$ is a composition factor of $H^j(\nu)$ for some $j \geq 0$ and $\nu \in X_\alpha^\lambda$.
\item If $\nu \in X_\alpha^\lambda $ then $\nu^+ \uparrow \lambda^+$ and $\nu^+ < \lambda^+$.
\end{enumerate}
\end{prop}
\begin{proof} (1) comes by combining the sequences (\ref{lseq1}) and (\ref{lseq2}). (2) is a weight calculation, see \cite{An80a}, Lemma 5.
\end{proof}

We shall now prove the theorem by induction with respect to the strong linkage ordering on $X^+ - \rho$. When $\lambda \in X^+-\rho$ we shall use the notation
$$ SL(<\lambda) = \{\mu \in X^+ -\rho \mid \mu \uparrow \lambda \text { and } \mu < \lambda \}.$$

Let $\lambda \in X^+-\rho$ and pick a reduced expression for $w_0$, $w_0 = s_Ns_{N-1} \cdots s_1$. Setting $\lambda_i = s_i s_{i-1} \cdots s_1 \cdot \lambda$ we have $\lambda_i^+ = \lambda$ for all $i$, and we get the following string of homomorphisms
\begin{equation}  \label{composite0}  { {H^N(\lambda_N)} \xrightarrow{c_N}  {H^{N-1}(\lambda_{N-1})} \xrightarrow{c_{N-1}} {\cdots} \xrightarrow{c_1} {H^0(\lambda_0)}}.
\end{equation}
Here $c_i = c_{\alpha_{i+1}}^{i+1}(\lambda_{i})$
 with $\alpha_{i+1}$ denoting the simple root corresponding to $s_{i+1}$ (note that $\langle \lambda_{i} + \rho, \alpha_{i+1}^\vee \rangle = \langle \lambda^+ + \rho, s_1s_2 \cdots s_{i}(\alpha_{i+1})^\vee \rangle \geq 0$ so that $ c_{\alpha_{i+1}}^{i+1}(\lambda_{i})$ exists).
 
 Actually,  we have such a string for all $j \in \Z$:
 \begin{equation}  \label{compositej}  { H^{N+j}(\lambda_N) \xrightarrow{c_N^j}  H^{N-1+j}(\lambda_{N-1}) \xrightarrow{c_{N-1}^j} \cdots \xrightarrow{c_1^j}  H^j(\lambda_0)},
\end{equation}
where $c_i^j = c_{\alpha_{i+1}}^{i+1+j}(\lambda_{i})$. Here we use the convention that $H^j(\nu) = 0$ when $j$ is negative. When $j$ is positive we notice that the first $j$ terms in (\ref{compositej}) are $0$ by (\ref{Grot}).

To start the induction assume that $\lambda$ is minimal in $X^+ - \rho$ with respect to the strong linkage relation. This means that $SL(<\lambda) = \emptyset$. But then all $X_{\alpha_i}^{\lambda_{i-1}}$ are also empty by Proposition {\ref{ker/cok} (2), i.e. the kernels and cokernels of all $c_i^j$ are $0$. We conclude that in this case the string (\ref{composite0}) consists of $N$ isomorphisms. This means that $H^N(\lambda_N) = H^i(\lambda_i) = H^0(\lambda_0)$, i.e. $H^i(\lambda_i) = L(\lambda)$ for all $i$. On the other hand, if $j \neq i$ then (\ref{compositej}) shows that $H^j(\lambda_i) = 0$.
In particular, the strong linkage principle certainly holds in this case.

Let now $\lambda \in X^+ - \rho$ be arbitrary and suppose Theorem \ref{SLP} holds for all $\mu \in SL(<\lambda)$. Suppose $L(\nu)$ is a composition factor of $H^{i+j}(\lambda_i)$ for some $0 \leq i \leq N$ and some $j \in \Z$. If $L(\mu)$ is a composition factor of one of the kernels or cokernels of one of the homomorphisms in the string (\ref{composite0}) we see from Proposition \ref{ker/cok} (1) that $L(\mu)$ is a composition factor of $H^t(\nu)$ for some $t$ and some $\nu \in X_{\alpha_r}^{\lambda_{r-1}}$ with $0 < r \leq N$. 
By (2) in the same proposition such $\nu$ all belong to $SL(< \lambda)$ and by our induction hypothesis we get in this case $\mu \uparrow \nu \uparrow \lambda$. On the other hand, if $L(\mu)$ is not a composition factor of any of these kernels or cokernels then the string (\ref{compositej}) shows,
 that we cannot have $j \neq 0$ and the string (\ref{composite0}) reveals that $L(\mu)$ must be in the image of the composite $c_1 \circ c_2 \circ \cdots \circ c_N$. But this composite is a homomorphism from $H^N(w_0 \cdot \lambda)$ to $H^0(\lambda)$. As $H^0(\lambda)$ has simple socle $L(\lambda)$ and (by Serre duality) $H^N(w_0 \cdot \lambda)$ has simple head $L(\lambda)$, we conclude that this image equals $L(\lambda)$. Hence in this case $\mu = \lambda$. 

Observe now that when we vary the reduced expression for $w_0$ we can for any $w \in W$ realize $w \cdot \lambda$ as one of the $\lambda_i$'s in our strings. As $W \cdot (X^+-\rho) = X$ we have thus  proved Theorem \ref{SLP}.

\begin{rem} \begin{enumerate}
\item As was the case with the original proof in \cite{An80a} the proof given above does not rely on Kempf's vanishing theorem (but the proof given in \cite{RAG}, Chapter II.6 does). In the preprint \cite{Anpre} we actually deduced Kempf's vanishing theorem from the strong linkage principle. However, unless we are in type $A$ this works only for $p \geq h$. The preprint was never published because it was shortly afterwards overtaken by \cite{An80b}, which requires no restrictions on $p$.
\item Since the publication of \cite{An80a} there have been some improvements and alternative proofs of the results, see e.g.  \cite{Wo1}, \cite{Wo2}, and \cite{Dot}. 
\item It should also be mentioned that we have a close analogue of the strong linkage principle for quantum groups with parameter being a root of $1$, see \cite{An03}.
\end{enumerate}
\end{rem}

\subsection{Applications}

In this section we deduce some consequences of the strong linkage principle and of its proof. The first is

\begin{cor} \label{LP} (The linkage principle)
Let $M$ be an indecomposable $G$-module. If $L(\lambda)$ and $L(\mu)$ are two composition factors of $M$ then $\mu \in W_p \cdot \lambda$.
\end{cor}

\begin{proof} It is standard to reduce this corollary to the following vanishing result
$$ \Ext^1_G(L(\mu), L(\lambda)) = 0 \text { unless } \mu \in W_p \cdot \lambda.$$
To check this we may assume that $\mu \not > \lambda$. In fact, if $\mu > \lambda$  we use the duality $\Ext_G^1(L(\mu), L(\lambda)) \simeq \Ext_G^1(L(\lambda)^*, L(\mu)^*) \simeq \Ext_G^1(L(-w_0 \cdot \lambda), L(-w_0 \cdot \mu))$ to bring us into the desired situation. 
We claim that $\Ext_G^1 (L(\mu), H^0(\lambda)) = 0$. Indeed, suppose we have an extension $0 \rightarrow H^0(\lambda) \rightarrow E \rightarrow L(\mu) \rightarrow 0$. 
Our assumption implies that $\lambda$ is a maximal weight of $E$. Hence the universal property (also known as Frobenius reciprocity, \cite{RAG} Proposition I.3.4) gives a splitting $E \rightarrow H^0(\lambda)$. 

Now the short exact sequence $0 \rightarrow L(\lambda) \rightarrow H^0(\lambda) \rightarrow H^0(\lambda)/L(\lambda) \rightarrow 0$ gives a surjection $\Hom_G(L(\mu), H^0(\lambda)/L(\lambda)) \rightarrow \Ext^1_G(L(\mu), L(\lambda))$. Hence if $\Ext^1_G(L(\mu), L(\lambda)) \neq 0$ we deduce that $L(\mu)$ is a composition factor of $H^0(\lambda)$. By Theorem \ref{SLP} this implies that $\mu$ is strongly linked to $\lambda$. In particular, $\mu \in W_p \cdot \lambda$.
\end{proof}

\subsection{On the $G$-structure of $H^1(\lambda)$}

The bound on the composition factors of cohomology modules of line bundles obtained in Theorem \ref{SLP} was the first general result on the $G$-structure of $H^i(\lambda)$. 
It holds for all $G, i, \lambda$, and $p$. Slightly before this result we obtained, however, another general result on the $G$-module structure of $H^1(\lambda)$, valid for all $G, \lambda$ and $p$, namely the following (recall that we have already given the complete vanishing behavior of $H^1(\lambda)$ in Theorem \ref{H1}).

\begin{thm} (\cite{An79b}, Theorem 3.5) All non-zero $ H^1(\lambda)$ have simple socles.
\end{thm}

\begin{proof} Here we offer a slightly simpler proof than the original one given in \cite{An79b}, Section 3. As observed in Section 4.3 if $H^1(\lambda) \neq 0$ then there exists a simple root $\alpha$ such that $\langle \lambda, \alpha^\vee \rangle < -1$ and 
\begin{equation} \label{red to H^0} H^1(\lambda) \simeq H^0(G/P_\alpha/B, H^1_\alpha(\lambda)).
\end{equation}
Let $\mu \in X^+$. Then we get from (\ref{red to H^0}) that $\Hom_G(L(\mu), H^1(\lambda)) \simeq \Hom_{P_\alpha} (L(\mu),H^1_\alpha(\lambda))$. Now we note that under the natural $P_\alpha$-homomorphism from $H^0(\mu)$ to $H^0_\alpha(\mu)$ the simple $G$-submodule $L(\mu) \subset H^0(\mu)$
maps onto the simple $P_\alpha$-submodule $ L_\alpha(\mu) \subset H^0_\alpha(\mu)$. Denoting by $K_\alpha(\mu) \subset L(\mu)$ the kernel of this map, we get the short exact sequence of $P_\alpha$-modules
\begin{equation} \label{simple G to P} 0 \to K_\alpha(\mu) \to L(\mu) \rightarrow L_\alpha(\mu) \to 0. \end{equation}
Here we have $\Hom_B(L(\mu), \nu) \simeq \Hom_G(L(\mu), H^0(\nu))$ is non-zero if and only if $\nu = \mu$. Therefore $\Hom_{P_\alpha}(L(\mu), H^1_\alpha(\lambda))$ is non-zero 
only if $\mu $ belong to the set of weights $\{\lambda + \alpha, \lambda + 2 \alpha, \cdots \lambda + (- \langle \lambda,\alpha^\vee \rangle -1)\alpha \}$ of $H^1_\alpha(\lambda)$. 
But for none of these $\mu$ do $K_\alpha(\mu)$ and $H^1_\alpha(\lambda)$ have any weights in common. This implies 
that $\Hom_{P_\alpha}(K_\alpha(\mu), H^1_\alpha(\lambda)) = 0$. 
We therefore get from (\ref{simple G to P} )
that $\Hom_{P_\alpha}(L_\alpha(\mu), H^1_{\alpha}(\lambda)) \simeq \Hom_{P_\alpha}(L(\mu), H^1_\alpha(\lambda))$. 

So we have reduced the theorem to the $SL_2$-case, where it is an easy computation.
\end{proof}

\begin{rem}
\begin{enumerate} 
\item The proof determines the highest weight $\mu$ of $\Soc_G(H^1(\lambda))$. It also shows that $L(\mu) = \Soc_G(H^1(\lambda))$  has multiplicity $1$ as a composition factor in $H^1(\lambda)$.
\item By Serre duality the theorem is equivalent to the statement that all non-zero $H^{N-1}(\lambda)$ have simple heads.
\item In the next section we shall show that if $\lambda$ is ``generic'' (see Definition 5.17 below for this condition) inside a Weyl chamber then there is a unique $i$ for which  $H^i(\lambda) \neq 0$. Moreover, this cohomology module has simple socle and simple head.  - This result fails badly in general. In fact, already for type $B_2$, we gave an example of a $\lambda$ for which $H^2(\lambda)$ splits into a direct sum of $2$ (simple) modules,  see the last paragraph in \cite{An80a}.
\end{enumerate} \end{rem}

\subsection{Filtrations and sum formulas for cohomology modules}
In Section 5.5 we shall recall from \cite{An83}, how one can obtain Jantzen-type filtrations and sum formulas for our cohomology modules $H^i(\lambda), i \geq 0, \; \lambda \in X$. When we take $i = N$ we recover the Weyl module case proved by Jantzen \cite{Ja77} for $p \geq h$. We sketch the proof taking advantage also of the methods developed in \cite{AKu}.

First we need to discuss $\Z$-versions of the cohomology modules. So we let $G_\Z$ denote the Chevalley group corresponding to $G$. This is a connected reductive algebraic group scheme over $\Z$ and $G$ is obtained from $G_\Z$ by base change from $\Z$ to $k$. More generally, if $A$ is any commutative ring we denote by $G_A$ the base change of $G_\Z$ to $A$ (so in particular $G_k = G$). We have likewise the subgroup schemes $T_A$ and $B_A$ corresponding to $T$ and $B$.

Let $M$ be a $B_A$-module (like all modules we shall consider $M$ is finitely generated over $A$). Then we denote by $H^i_A(M)$  the
$i$'th cohomology module of the vector bundle on $G_A/B_A$ associated to $M$. Alternatively, $H_A^i$ is also the $i$'th right derived functor of induction from $B_A$ to $G_A$. Then we have the following universal coefficient theorem, cf. \cite{Bou}, Ch §4.

\begin{thm} \label{UC} Let $M$ be a $B_\Z$-module which is free (of finite rank) over $\Z$. Then for each $i\geq0$ we have a short exact sequence of $A$-modules
$$ 0 \to H^i_\Z(M) \otimes_\Z A \to H^i_A(M \otimes_\Z A) \to \Tor_1^\Z(H^{i+1}_\Z(M), A) \to 0.$$
\end{thm}

In particular, we may in Theorem \ref{UC} take $M$ to be the rank $1$ $\Z$-module $\Z$ with $B_\Z$-structure given by $\lambda \in X$. We shall abuse notation and write just $\lambda$ for this $B_\Z$-module. Moreover, we let $H_t^i(\lambda)$ denote the torsion part of $H^i_\Z(\lambda)$ for any $i \geq 0$, and we set $H_f^i(\lambda) = H_\Z^i(\lambda)/H_t^i(\lambda)$, the free quotient of $H_\Z^i(\lambda)$. Then Theorem \ref{UC} combined with Bott's theorem (Theorem \ref{Bott}) give the following results.

\begin{cor}\label{torsion} Let $\lambda \in X$. 
\begin{enumerate}
\item If $\lambda$ is singular then $H^i(\lambda) = H^i_t(\lambda)$ for all $i$.
\item If $\lambda$ is regular and $w \in W$ is the unique element with $w \cdot \lambda \in X^+$ then 
$$ H^i_\Z(\lambda) = H_t^i(\lambda) \text { for all } i \neq \ell(w),$$
and 
$$ H_f^{\ell(w)}(\lambda) \text { has rank equal to } \dim_\Q H^0_\Q(w\cdot \lambda),$$
\item for each $i \geq 0$ we have a short exact sequence 
$$  0 \to H^i_\Z(\lambda) \otimes_\Z k \to H^i(\lambda \otimes_\Z k) \to \Tor_1^\Z(H^{i+1}_\Z(\lambda), k) \to 0.$$
\end{enumerate}
\end{cor}
Combining Theorem \ref{UC} with Grothendieck vanishing (Theorem \ref{Grot}), respectively Kempf vanishing (Theorem \ref{Kempf}) we get

\begin{cor} \label{free} Let $\lambda \in X$.
 \begin{enumerate}
\item $H^i_\Z(\lambda) = 0$ for all $i > N$, and $H^N_\Z(\lambda)$ is free over $\Z$ of rank equal to $\dim_\Q H^0_\Q(w_0\cdot\lambda)$. Moreover, $H^N_\Z(\lambda) = 0 = H^N_k(\lambda)$ unless $ \lambda \in X^-$.
\item If $\lambda \in X^+$ then $H^i_\Z(\lambda) = 0$ for all $i > 0$ and $H^0_\Z(\lambda)$ is free of rank equal to $\dim_\Q H^0_\Q(\lambda)$.
\end{enumerate}
\end{cor}

We shall need some more notation in order to state and prove the sum formulas. Recall that $p$ is the characteristic of $k$. If $n \in \Z$ is non-zero we shall denote by $\nu_p(n)$ the highest exponent $s$ of $p$ such that $p^s$ divides $n$. Likewise, if $M$ is a finite abelian group of order $|M|$ we write $\nu_p(M) = \nu_p(|M|)$. Then $\nu_p(M)$ is also the length of $M$ as $\Z$-module. 

Suppose $M$ is a finite $T_\Z$-module. Then we define 
$$ \Char^p(M) = \sum_{\lambda \in X} \nu_p(M_\lambda) e^\lambda \in \Z[X]$$
and call this the $p$-character of $M$.
Note that $\Char^p$ is additive on short exact sequences of finite $T_\Z$-modules.

Finally, if $M$ is a finite $B_\Z$-module then we define
$$ E^p (M) = \sum_i (-1)^i \Char^p(H^i_\Z (M)).$$
$E^p$ is also additive on short exact sequences of finite $B_\Z$-modules.

The following proposition is the key to calculating $E^p(M)$. For any $\lambda \in X$ we denote by $\chi(\lambda)$ the Weyl character at $\lambda$, i.e. $\chi(\lambda) = \sum_{i=0}^N (-1)^i \Char H^i(\lambda)$.  We have $\chi (w \cdot \lambda) = (-1)^{\ell(w)} \chi(\lambda)$ for any $w \in W$.

\begin{prop} \label{Z/n}
Let $n \in \Z_{>0}$ and $\lambda \in X$. Then
$$ E^p (\lambda \otimes_\Z \Z/n) = \nu_p(n) \chi(\lambda).$$
\end{prop}

\begin{proof} The short exact sequence of $B_\Z$-modules $0 \to \lambda \xrightarrow {n} \lambda \to \lambda \otimes_\Z \Z/n \to 0$ gives the long exact cohomology sequence
$$ 0 \to H^0_\Z(\lambda) \xrightarrow {n} H^0_\Z(\lambda) \to H^0_\Z(\lambda \otimes \Z/n) \to H^1_\Z(\lambda) \xrightarrow {n}  H^1_\Z(\lambda) \to  H^1_\Z(\lambda \otimes \Z/n) \to \cdots $$
If $\lambda$ is singular, then this sequence consists entirely of finite $\Z$-modules and the result is clear (note that $\chi(\lambda) = 0$ for $\lambda$ singular). If $\lambda$ is regular, there exists (by Corollary \ref{torsion} (1)) a unique $i$ such that $H^i_\Z(\lambda)$ is not finite. For this $i$ we have the following diagram 
\[
\begin{tikzcd}
H^{i-1}_\Z(\lambda \otimes \Z/n) \arrow[d] \arrow[r]
& H^i_t(\lambda) \arrow[r]  \arrow[d]
&H^i_t(\lambda) \arrow[d] \arrow[r] 
& Q_t \arrow[d]\\
H^{i-1}_\Z(\lambda \otimes \Z/n) \arrow[r] \arrow[d]
&H^i_\Z(\lambda) \arrow[r]\arrow[d]
&H^i_\Z(\lambda) \arrow[r]\arrow[d]
&Q \arrow[d]\\
0 \arrow[r] 
&H^i_f(\lambda) \arrow[r]
& H^i_f(\lambda) \arrow[r] &Q_f.
\end{tikzcd}
\]
Here the top left vertical arrow is the identity map, the three middle horizontal arrows are multiplication by $n$, and $Q_t$,  $Q$, and $Q_f$ denote the cokernels of multiplication by $n$ on the modules in question.

Now the top sequence in the above diagram is the end of an exact sequence of torsion modules. It gives 
$$\Char^p(Q_t) =
 \sum_{j\geq 0}(-1)^j \Char^p (H^{i-1-j}(\lambda \otimes_\Z \Z/n)).$$ 
 Likewise $Q$ is the first term in the exact sequence
 $$ 0 \to Q \to H^{i}_\Z(\lambda \otimes_\Z \Z/n) \to H^{i+1}_\Z (\lambda) \to H^{i+1}_\Z (\lambda) \to  H^{i}_\Z(\lambda \otimes_\Z \Z/n) \to \cdots , $$
 which also consists of torsion modules, cf. Corollary \ref{torsion}. Therefore we get
  $$\Char^p (Q) = \sum_{j \geq 0} (-1)^j \Char^p(H^{i+j}(\lambda \otimes _\Z \Z/n)).$$ 
  As the right column in the above diagram is a short exact sequence we deduce that $\Char (Q_f) = \Char^p(Q) - \Char^p(Q_t) = (-1)^i \sum_{j } (-1)^j \Char^p(H^j(\lambda \otimes_\Z \Z/n)) = (-1)^i E^p (\lambda \otimes_\Z \Z/n)$.  The proposition thus follows as the above formulas show that $Q_f = 
 H^i_f(\lambda) \otimes_\Z \Z/n$ has $p$-character equal to $ \nu_p(n) (-1)^i\chi(\lambda)$. 

\end{proof}

Let now $\lambda \in X^+$ and set 
$$ \Delta(\lambda) = H^N(w_0\cdot \lambda) \text { and } \nabla(\lambda) = H^0(\lambda).$$
These are the Weyl module and the dual Weyl module with highest weight $\lambda$. They both have $\Z$-forms, namely $\Delta_\Z(\lambda) 
= H^N_\Z(w_0 \cdot \lambda)$, respectively $\nabla_\Z(\lambda) = H^0_\Z(\lambda)$. According to Corollary \ref{free} we have $\Delta_\Z(\lambda) \otimes_\Z \Q \simeq H^N_\Q(w_0 \cdot \lambda) \simeq H^0_\Q(\lambda) \simeq \nabla_\Z(\lambda) \otimes \Q$.

Choose a reduced expression $w_0 = s_N \cdots s_1s_0$ for $w_0$ and as in Section 5.1 set $\lambda_i = s_i s_{i-1} \cdots s_1 \cdot \lambda$. We then get a long exact sequence analogous to (\ref{lseq1}) 
$$ \Delta_\Z(\lambda) = H^N_\Z(\lambda_N) \to \cdots H^{i+1}_\Z(\lambda_{i+1}) \to H^i_\Z(\lambda_i) \to \cdots \to H^0_\Z(\lambda_0) = \nabla_\Z(\lambda).$$

In the notation above $H^i_t(\lambda_i)$ is the torsion submodule of $H^i_\Z(\lambda_i)$ and $H_f^i(\lambda_i )= H^i_\Z(\lambda_i)/H^i_t(\lambda_i)$ is the free quotient. If $j \neq i$ then $H^j_\Z(\lambda_i)$ is a torsion module, cf. Corollary \ref{torsion}. 
Clearly, each of the homomorphisms $c_\Z^i(\lambda): H^{i+1}_\Z(\lambda_{i+1}) \to H_\Z^i(\lambda_i)$ induces a homomorphism $c_f^i(\lambda) : H_f^{i}(\lambda_{i}) \to H_f^{i-1}(\lambda_{i-1})$. We set $c_f(\lambda) = c_f^1(\lambda) \circ c_f^2(\lambda) \circ \cdots \circ c_f^N (\lambda): \Delta_\Z(\lambda) \to \nabla_\Z(\lambda)$.
As in Section 5.1 we see that on the $\lambda$ weight space each $c_f^i(\lambda)$ is an isomorphism. Hence the same is true for $c_f(\lambda)$. This means in particular, that up to signs $c_f(\lambda)$ is independent of the reduced expression we have chosen for $w_0$, and it implies that $c_f(\lambda) \otimes_\Z k$ is a non-zero homomorphism from $\Delta (\lambda)$ to $\nabla(\lambda)$.

\begin{thm} \label{JSF} (The Jantzen filtration and sum formula for Weyl modules.)
Let $\lambda \in X^+$. The homomorphism $c_f(\lambda): \Delta_\Z(\lambda) \to \nabla_\Z(\lambda)$ induce a filtration  of $\Delta (\lambda)$
 $$ 0 = \Delta^{r+1}(\lambda) \subset  \Delta^{r}(\lambda) \subset \cdots \subset  \Delta^{1}(\lambda) \subset  \Delta^{0}(\lambda)= \Delta (\lambda),$$
 which satisfies
 \begin{enumerate} 
 \item $ \Delta (\lambda)/ \Delta^{1}(\lambda) = L(\lambda)$
 \item $ \sum_{j=1}^r \Char  \Delta^{j}(\lambda) =
 - \sum_{\beta \in R^+} \sum_{0 < m  < \langle \lambda + \rho, \beta^\vee \rangle } \nu_p(m) \chi(\lambda - m\beta)$.
 \end{enumerate}
 \end{thm}
 
 \begin{proof} We set $\Delta_\Z^j(\lambda) = c_f(\lambda)^{-1}(p^r \nabla_\Z(\lambda))$. Denoting by $\pi: \Delta_\Z(\lambda) \to \Delta(\lambda)$ the natural homomorphism we then define a filtration of $\Delta (\lambda)$ by setting
 $$  \Delta^{j}(\lambda) = \spa_k (\pi(\Delta_\Z^j(\lambda)) \subset \Delta (\lambda).$$
 This is clearly a finite filtration, i.e. there exists an $r\geq 0$ such that  $\Delta^{r+1}(\lambda) = 0$. Moreover, by the observation just above the theorem we see that (1) holds, because $\Hom_G(\Delta(\lambda), \nabla(\lambda)) = k$ and any non-zero $G$-homomorphism (like $c_f(\lambda) \otimes_\Z 1$) from $\Delta(\lambda)$ to $\nabla(\lambda)$ has image $L(\lambda)$.
 
To prove (2) we notice first that the left hand side of (2) equals $\Char^p(Q)$ where $Q$ is the cokernel of $c_f(\lambda)$. This is a well known computation, see e.g. \cite{RAG}, Section II.8.11. Now each $c_f^i(\lambda)$ is injective, 
because by Bott's theorem it becomes an isomorphism after  tensoring with $\Q$. Hence the additivity of $\Char^p$ gives $\Char^p (Q) = \sum_{i=1}^N \Char^p(Q_f^i)$ where $Q_f^i$ denotes the cokernel of $c_f^i(\lambda)$. 
Letting $Q^i$ be the cokernel of $c_\Z^i(\lambda)$ and $Q^i_t$ be the cokernel of the restriction of $c_\Z^i(\lambda)$ to $H_t^i(\lambda_i)$ we have
$$ \Char^p(Q^i_f) = \Char^p(Q^i) - \Char^p(Q_t^i).$$
We shall now compute the terms on the right hand side of this equation by using the techniques from the proof of Proposition \ref{Z/n} combined with a $\Z$-variation of the methods from Section 5.1.

Let $\mu \in X$ and $\alpha \in S$ satisfy $\langle \mu, \alpha^\vee \rangle \geq 0$. When working over $\Z$ the 4 short exact $B$-sequences in Section 5.1 ``simplify'' to the following 3 short exact $B_\Z$-sequences (using analogous notation to the one in Section 5.1)
\begin{equation} \label{Z-1}
0 \to K_{\alpha, \Z}^\mu \to H^0_{\alpha, \Z}(\mu + \rho) \otimes_\Z -\rho \to \mu \to 0,
\end{equation}
\begin{equation} \label{Z-2} 0 \to s_\alpha \cdot \mu \to K_{\alpha. \Z}^\mu \to V_{\alpha, \Z}^\mu \to 0,
\end{equation}
and
\begin{equation}\label{Z-3}
0 \to V_{\alpha, \Z}^\mu \to H^0_{\alpha, \Z}(\mu -\alpha + \rho) \otimes_\Z -\rho \to Q_{\alpha, \Z}^\mu \to 0.
\end{equation}
The difference when working over $\Z$ is that here the natural $B_\Z$-homomorphism  $ V_{\alpha, \Z}^\mu \to H^0_{\alpha, \Z}(\mu -\alpha + \rho) \otimes_\Z -\rho$ is injective.

We shall need the following observation (an $SL_2$-computation) about the torsion module $Q_{\alpha, \Z}^\mu$.
\begin{equation} \label{cok} \text {The weight spaces of  $Q_{\alpha, \Z}^\mu$ are $\Z/j \otimes_\Z (\mu - j \alpha), \; j = 1, 2, \cdots , \langle \mu + \rho, \alpha^\vee \rangle - 1.$}
\end{equation}
 
 Just as in Section 5.1 we see from (\ref{Z-1}) that $H^i(\mu) \simeq H^{i+1} (K_{\alpha, \Z}^\mu)$ and from  (\ref{Z-3}) that $H_\Z^{i+1} (V_{\alpha, \Z} ^\mu ) \simeq H^i_\Z(Q_{\alpha, \Z}^\mu)$ . 
 
 Therefore we get by combining this with (\ref{Z-2}) the long exact sequence 
 $$ \cdots \to H^{i+1}_\Z(s_\alpha \cdot \mu) \to H_\Z^i(\mu) \to H^i_{\Z} (Q_{\alpha, \Z}^\mu) \to \cdots .$$
 
 We now apply this sequence with $\mu = \lambda_i$ and $\alpha = \alpha_{i+1}$, the simple root associated to $s_{i+1}$. Remembering that  $H^{j+1}_\Z(\lambda_{i+1})$ and $H^j_\Z(\lambda_i)$  are torsion modules for all $j \neq i$ we see that $Q_t^{i+1}$ is the last term in the exact sequence 
$$ \cdots \to H^{i}_t(\lambda_{i+1}) \to H^{i-1}_t(\lambda_i) \to H^{i-1}_\Z(Q_{\alpha_i, \Z}^{\lambda_i}) \to H^{i+1}_t (\lambda_{i+1}) \to  H^{i}_t (\lambda_{i}) \to Q_t^{i+1} \to 0,$$
 whereas $Q^{i+1}$ is the first term in the exact sequence
 $$ 0 \to Q^{i+1} \to H^{i}_\Z(Q_{\alpha_i, \Z}^{\lambda_i})   \to   H^{i+2}_t (\lambda_{i+1}) \to  H^{i+1}_\Z(\lambda_i) \to \cdots .$$
 These sequences consist entirely of torsion modules and we get
 $$ \Char^p (Q_t^{i+1}) = \sum_{j \geq 0} (-1)^j ( \Char^p(H^{i-j}_t(\lambda_i)) - \Char^p(H^{i+1-j}_t(\lambda_{i+1})) +  \Char^p(H^{i-1-j}_t(Q_{\alpha_{i+1}, \Z}^{\lambda_i}))), $$
while
$$  \Char^p (Q^{i+1}) = \sum_{j \geq 0} (-1)^j (\Char^p(H^{i+j}_t(Q_{\alpha_{i+1}, \Z}^{\lambda_i})) -  \Char^p(H^{i+2+j}_t(\lambda_{i+1})) +  \Char^p(H^{i+1+j}_t(\lambda_i))).$$
We get from this
$$ \Char^p Q_f^i = (-1)^i(E^p(Q_{\alpha_{i+1}, \Z}^{\lambda_i}) + \sum_{j \geq 0} (-1)^j (\Char^p(H^j_t(\lambda_{i+1} )) - \Char^p(H_t^j (\lambda_i) )).$$ 
 
 By (\ref{cok}) and Proposition \ref{Z/n} we have 
 $$ E^p(Q_{\alpha_{i+1}, \Z}^{\lambda_i}) = \sum_{j=1}^{\langle \lambda_i + \rho, \alpha_{i+1}^\vee \rangle -1} \nu_p(j)  \chi(\lambda_i - j \alpha_{i+1}) = (-1)^i \sum_{j=1}^{\langle \lambda + \rho, \beta_{i+1}^\vee \rangle -1} \nu_p(j)  \chi(\lambda - j \beta_{i+1}), $$
 where $\beta_{i+1} = s_1s_2 \cdots s_i(\alpha_{i+1})$.
 
 Write now $E_t^p(i) = \sum _{j \geq 0} (-1)^j \Char^p(H^j_t(\lambda_i)), \; i = 0, 1, \cdots ,N-1$. 
 Noting that $\{\beta_1, \beta_2, \cdots , \beta_N\} = R^+$  we get from the above
 $$ \Char^p(Q) = \sum_{i=0}^{N-1} \Char^p(Q_f^{i+1}) = - \sum_{\beta \in R^+} \sum_{j=1}^{\langle \lambda + \rho, \beta^\vee \rangle - 1} \nu_p(j) \chi(\lambda - j \beta) + \sum_{i=0}^{N-1} (E_t^p(i+1) - E_t^p(i))$$
 Here the last sum equals $E_t^p(N) -E_t^p(0) = 0$ because both $H_\Z^N(\lambda_N)$ and $H_\Z^0(\lambda_0)$ are torsionfree, see Corollary \ref{free}. Thus we have established (2).
 \end{proof}
 
 \begin{rem}
 In low rank, e.g. rank $2$ or type $A_3$, Theorem \ref{JSF} combined with the translation principle \cite{RAG}, Chapter II.7 give all the irreducible characters. The result has also proved very useful for finding or at least limiting particular composition factor multiplicities in Weyl modules in many other cases, e.g. for small characteristics (note that the above proof requires no restrictions on $p$) or for special highest weights.
 \end{rem}
 \vskip .5 cm
 Consider again $\mu \in X$ and $\alpha \in S$ with $\langle \mu, \alpha^\vee \rangle \geq 0$. In addition to the unique (up to sign) non-zero $P_{\alpha, \Z}$-homomorphism $c_\alpha^\mu : H_{\alpha, \Z}^1(s_\alpha \cdot \mu) \to H^0_{\alpha, \Z} (\mu)$ we have
 a similar homomorphism $\tilde c_{\alpha}^\mu :  H^0_{\alpha, \Z} (\mu) \to  H_{\alpha, \Z}^1(s_\alpha \cdot \mu)$, 
 which is unique once we require $\tilde c_{\alpha}^\mu \circ  c_{\alpha}^\mu = \langle \mu, \alpha^\vee \rangle ! Id_{H^1_{\alpha, \Z} (s_\alpha \cdot \mu)}$.  
 Just like $c_\alpha^\mu$ gives rise to $G_\Z$-homomorphisms $H_\Z^{i+1}(s_\alpha \cdot \mu) \to H_\Z^{i}( \mu)$ for all $i \geq 0$, so the homomorphism $\tilde c_\alpha^\mu$ induces $G_\Z$-homomorphisms in the reverse directions. Composing these maps according to an appropriate reduced expression for $w_0$ one obtains for each $w \in W$ and $\lambda \in X^+$ a $G_\Z$-homomorphism
 $$ c_f(w, \lambda) : H_f^{\ell(w)} (w\cdot \lambda) \to H_f^{N-\ell(w)}(w_0w \cdot \lambda).$$
 Arguing as in the proof of Theorem \ref{JSF} (cf. \cite{An83}) we then get
 \begin{thm} \label{SF} (Filtrations and sum formulas of cohomology modules)
 Let $\lambda \in X^+$ and set $r_\beta = \min \{r \in Z_{\geq 0} | p^r \leq \langle \lambda + \rho, \beta^\vee \rangle \}$ for each $\beta \in R^+$.  Then for   $w \in W$ the module $H_f^{\ell(w)}(w \cdot \lambda) \otimes_\Z k$ has a filtration 
 $$ 0 = F^{s+1}(w\cdot \lambda) \subset  F^{s}(w\cdot \lambda) \subset \cdots \subset  F^{1}(w\cdot \lambda) \subset  F^{0}(w\cdot \lambda) = H_f^{\ell(w)}(w \cdot \lambda) \otimes_\Z k$$
 which satisfies the sum formula
 $$\sum_{j\geq 1}\Char (F^j(w\cdot \lambda)) = (\sum_{\alpha \in R^+\cap w^{-1}R^+} r_\alpha) \chi(\lambda) + \sum_{\beta \in R^+} sgn(w(\beta)) \sum_{0 < m < \langle \lambda + \rho, \beta^\vee \rangle} \nu_p(m) \chi(\lambda - m\beta)$$
 $$  + (-1)^{\ell (w_0w)}(E_t^p(w_0w \cdot \lambda) - E_t^p(w\cdot \lambda)).$$
 \end{thm}
 
 \begin{rem} \label{problems}
 Unfortunately, this theorem gives in general only  a filtration of a subquotient  of $H^{\ell(w)}(w\cdot \lambda)$. Moreover, in contrast to Theorem \ref{JSF} it contains no statement (1) about the top term of the filtration. Finally, in Theorem \ref{JSF} the right hand side of the sum formula is a sum of well known terms (Weyl characters), but in the present theorem the right hand side has two additional terms, which are not known in general. As we shall see in the next section (cf. in particular  Corollary \ref{GSF}) we can ``repair'' these deficiencies for {\it generic} weights. Also see \cite{An86a} for further related results on the ``torsion part'' of $H^i(w\cdot \lambda)$.
 
 \end{rem}

\subsection{Cohomology of line bundles with generic weights}
It turns out that if $\mu$ lies ``far away'' from the walls of the Weyl chambers in $X$, then the cohomology $H^i(\mu)$ behaves much better than in general (both when it comes to its vanishing behavior, and when it concerns its $G$-module structure). In this section we shall briefly discuss  this case. We follow \cite{An86b} where further details and results can be found.

Although parts of the statements below are true under milder assumptions we shall (as in \cite{An86b}) say that a dominant weight is {\it generic} if it satisfies the following conditions.

\begin{defn} Let $\lambda \in X^+ $ and write $\lambda = \lambda^0 + p^n \lambda^1$ with $\lambda^0 \in X_n$ and $\lambda^1 \in X^+$. We say that $\lambda$ is {\it generic}  if it satisfies the conditions
$$ 6 (h-1) \leq \langle \lambda^1, \beta^\vee \rangle \leq p- 6(h-1) \text { for all } \beta \in R^+.$$
\end{defn}
Note that {\it generic} weights exist only when $p > 12(h-1)$.

\vskip .5 cm

If $H$ is a subgroup(scheme) in $G$ and $M$ is an $H$-module we denote by $\Soc_H (M)$, respectively by $\Hd_H (M)$, the socle, respectively the head, of $M$. When $\lambda \in X^+$ and $w \in W$ we denote by $\lambda^w$ the weight determined by $\lambda^w = (w\cdot \lambda)^0 + p^n w^{-1} \cdot ((w\cdot \lambda)^1)$.

\begin{thm} \label{generic} (\cite{An86b}, Theorem 2.1 and Theorem 2.2) Let $\lambda \in X^+$ be {\it generic}. Then for any $w \in W$ we have
\begin{enumerate}
\item $H^i(w \cdot \lambda) = 0$ for all $i \neq \ell(w)$,
\item $\Soc_{G_n}(H^{\ell(w)}(w \cdot \lambda)) = L(\lambda^w) = \Soc_G(H^{\ell(w)}(w\cdot \lambda))$,
\item  $\Hd_{G_n}(H^{\ell(w)}(w \cdot \lambda)) = L(\lambda^{w_0w}) = \Hd_G(H^{\ell(w)}(w\cdot \lambda))$.
\end{enumerate}
\end{thm}
\begin{rem} \begin{enumerate}
\item Theorem \ref{generic} (1)  says that the vanishing statement in Bott's theorem holds in characteristic $p>0$, when $\lambda$ is {\it generic} (note that this condition depends on $p$). This was first observed by Cline, Parshall and Scott in the appendix of \cite{CPS}.
\item The proof of this theorem uses the isomorphisms $H^i(\mu) \simeq H^i(G/G_nB, H^0(G_nB/B, \mu))$ resulting from (\ref{higher vanishing}). 
As all $G_nB$-composition factors $H^0(G/G_nB, \mu)$  have highest weights lying ``close to'' $\mu$ we get information about $H^i(\mu)$ by studying a $G_nB$-composition series of $H^0(G/G_nB, \mu)$. In addition to the vanishing result in (1) 
this may also be used to find the $G$-composition factors for $H^{\ell(w)}(w \cdot \lambda)$ in terms of the $G_nB$-composition factors of $H^0(G/G_nB, w\cdot \lambda)$, see \cite{CPS}, Proposition A.1(b) and \cite{An86b}, Theorem 2.1. The vanishing result in Theorem \ref{generic} (1) ensures that the set of composition factors of $H^{\ell(w)} (w \cdot \lambda)
$ is independent of $w$. 
\item The proof of the results in (2) and (3) on the socles and heads of $H^{\ell(w)}(w \cdot \lambda))$ relies on the fact that the corresponding indecomposable injective modules for $G_n$ have a $G_nB$-structure. This is known to be true for $p \geq 2(h-1)$, see \cite{RAG} , Section II.11.11, a condition much weaker than our condition for {\it generic} weights to exist.  Furthermore, the indecomposable injective $G_n$-modules  have $G_nB$-filtrations with quotients equal to $H^0(G_nB/B, \mu)$ for certain $\mu \in X$, which are again ``close enough'' to $w\cdot \lambda$. For details, see the proof of Theorem 2.2 in \cite{An86b}.
\item J.E. Humphreys conjectured in \cite{Hu86}, Conjecture p. 178, that if $\mu$ is in sufficiently general position, then it will always have a unique non-vanishing cohomology module, say $H^j(\mu)$, and this module will have simple head and socle.  Theorem \ref{generic} proves this conjecture (with the above {\it generic} conditions on $\lambda$ being our way of ensuring that all $\mu \in W\cdot \lambda$ are in ``sufficient general position'').
\item J.E. Humphreys presented in \cite{Hu86}, Section 2 several other conjectures on the module structures of the cohomology modules $H^i(\mu)$.
For instance, he suggested that when $\lambda \in X^+$ is in ``sufficient general position'' then the $G$-module structure of $H^{\ell(w)} (w \cdot \ \lambda)$ is ``similar''  to that of the dual Weyl module $H^0(\lambda)$. 
Note that by (1) it follows that for $\lambda$ {\it generic},  all $H^{\ell(w)}(w \cdot \lambda)$ have the same composition factors (counted with multiplicity). Jim suggested that these composition factors are $scrambled$ in a predictable way for each $w$. 
He also speculated that when moving $\lambda$ from a sufficiently general position towards one or more walls of the dominant chamber,  the ``missing'' occurrences of certain composition factors in $H^0(\lambda)$ that one observes  will happen in this process, should be explained by the non-standard behavior of the higher cohomology modules for the various $w \cdot \lambda$. These predictions/speculations are very much still open today (even for very large primes). See \cite{DS}, \cite{Li90} and \cite{Li91} for some partial results.
\item For some alternative ways of obtaining explicit descriptions of higher cohomology modules for line bundles on flag varieties see \cite{Don}, \cite{Liu}, and \cite{LP}.
\end{enumerate}
\end{rem} 
In addition to trying to answer some of Jim's questions and predictions we had another motivation for proving Theorem \ref{generic}, namely we wanted to improve on Theorem \ref{SF}. More precisely, we wanted to address the problems mentioned in Remark \ref{problems}. This consequence of Theorem \ref{generic} is (we use notation as above as well as from Theorem \ref{JSF}):
\begin{cor} \label{GSF}
 Let $\lambda \in X^+$ be {\it generic} and suppose $w \in W$. Then $H^{\ell(w)}(w \cdot \lambda)$ has a filtration 
 $$ 0 = F^{s+1}(w\cdot \lambda) \subset  F^{s}(w\cdot \lambda) \subset \cdots \subset  F^{1}(w\cdot \lambda) \subset  F^{0}(w\cdot \lambda) = H^{\ell(w)}(w \cdot \lambda)$$
 which satisfies 
 \begin{enumerate} 
 \item $H^{\ell(w)}(w\cdot\lambda)/ F^{1}(w\cdot \lambda) = L(\lambda^{w_0w})$,
 \item $\sum_{j\geq 1}\Char (F^j(w\cdot \lambda)) = (\sum_{\alpha \in R^+\cap w^{-1}R^+} r_\alpha)  \chi(\lambda) + \\  \hskip 2 cm  \sum_{\beta \in R^+} sgn(w(\beta)) \sum_{0 < m < \langle \lambda + \rho, \beta^\vee \rangle} \nu_p(m) \chi(\lambda - m\beta).$
\end{enumerate}
 \end{cor}

\vskip 1 cm
\end{document}